\documentclass[11pt,a4paper]{article}

\usepackage[utf8]{inputenc}
\usepackage[english]{babel}
\usepackage{amsmath}
\usepackage{amsthm}
\usepackage{amsfonts}
\usepackage{amssymb,bm}
\usepackage[pdftex]{color,graphicx}
\usepackage{mathrsfs}  
\usepackage[expansion]{microtype}
\usepackage{ stmaryrd }

\usepackage[pdftex,colorlinks=true]{hyperref}

\hypersetup{
    pdfauthor   = {Enriquez Faraud Ménard Noiry},%
    pdftitle    = {DFS supercritical Configuration Model}}

\usepackage[font={small,sf}, labelfont={sf,bf}, margin=1cm]{caption}

\newtheorem{lemme}{Lemma}
\newtheorem{prop}{Proposition}

\newtheorem{theorem}{Theorem}

\newtheorem{remark}{Remark}

\newcommand\E{ \mathbb{E} }

\newcommand\geom{\mathfrak{e}}

\title{Long induced paths in a configuration model}
 \author{Nathana\"el Enriquez, Gabriel Faraud, Laurent Ménard and Nathan Noiry}
\date{}

\begin{document}
\maketitle

\abstract{
In an article published in 1987 in Combinatorica \cite{MR918397}, Frieze and Jackson established a lower bound on the length of the longest induced path (and cycle) in a sparse random graph. Their bound is obtained through a rough analysis of a greedy algorithm. In the present work, we provide a sharp asymptotic for the length of the induced path constructed by their algorithm. To this end, we introduce an alternative algorithm that builds the same induced path and whose analysis falls into the framework of a previous work by the authors on depth-first exploration of a configuration model \cite{EFMN}. We also analyze an extension of our algorithm that mixes depth-first and breadth-first explorations and generates $m$-induced paths.
\bigskip

\bigskip

}

\section{Introduction}

In this paper, we are interested in the existence of long induced cycles and long induced paths in random (multi-)graphs. An induced path of length $k \geq 1$ in a (multi-)graph $G=(V,E)$ is a sequence $v_1, \ldots, v_k$ of distinct vertices of the graph such that, for any $1 \leq i,j \leq k$, $v_i$ and $v_j$ are neighbors in the graph $\mathrm{G}$ (\emph{i.e.} $\{v_i,v_j\} \in E$) if and only if $|i-j|=1$. Similarly, an induced cycle is a cycle of distinct vertices of the graph such that two non-consecutive vertices are not linked by an edge. Induced cycles are often also called {\it holes} of the graph. Our main result, Theorem \ref{theo:main}, is a lower bound on the length of induced paths and cycles for random graphs constructed by the configuration model in the supercritical sparse regime. This length is linear in the size of the graph.
\bigskip

The configuration model was introduced by Bollob\'as in 1980 \cite{B} and can be defined as follows:
Let $N\geq 1$ be an integer and let $d_1, \ldots, d_N \in \mathbb{Z}_+$ be such that $d_1 + \cdots +d_N$ is even. We interpret $d_i$ as a number of half-edges attached to vertex $i$. Then, the configuration model $\mathscr{C}( (d_i)_{1 \leq i \leq N})$ associated to the sequence $(d_i)_{1 \leq i \leq N}$ is the random multigraph with vertex set $\{1, \ldots ,N\}$ obtained by a uniform matching of these half-edges. If $d_1 + \cdots + d_N$ is odd, we change $d_N$ into $d_N+1$ and do the same construction.

We are going to study sequences of configuration models whose degree sequences $\mathbf{d}^{(N)}$ satisfy classical hypothesis. The first assumption corresponds to the sparsity of the graphs:
\begin{itemize}
\item For every $N\geq 1$, $\mathbf{d}^ {(N)} = (d_1^{(N)} , \ldots, d_N^{(N)}) \in \mathbb{Z}_+^N$, and there exists $\bm \pi$ a probability measure on $\mathbb{Z}_+$ with finite second moment such that
\[  \forall k \geq 0, \quad \frac{1}{N} \sum\limits_{ j=1}^N \mathbf{1}_{ d_j^{(N)} = k}  \underset{ N \rightarrow +\infty}{ \longrightarrow} \bm \pi( \{k\}).   \tag{{\bf A1}}\]
\end{itemize}

Such graphs are known to exhibit a phase transition for the existence of a unique macroscopic connected component depending on the properties of the limiting distribution $\bm \pi$, see \cite{MR1,MR2,MR2490288}. Our second assumption is that our configuration models are supercritical, meaning that they have a unique giant component, so that long induced paths and cycles can have linear length. To state this assumption, we need the following notation: for every $s\in [0,1]$ define
\[
f_{\bm \pi}(s) = \sum\limits_{i \geq 0} \bm \pi(\{i\}) s^i \quad \text{and} \quad \hat f_{\bm \pi}(s) =  \frac{f'_{\bm \pi}(s)}{f'_{\bm \pi}(1)}.
\]
Our second assumption is then:

\begin{itemize}
\item The probability measure $\bm \pi$ is supercritical in the following sense:
\[   \hat f_{\bm \pi}'(1)  >1 .  \tag{{\bf A2}}\]
\end{itemize}

Denote by $\rho_{\bm \pi}$ the smallest positive solution in $(0,1]$ of the fixed point equation
\begin{equation} \label{eq:rho2}
1 - \rho_{\bm \pi} =  \hat f_{\bm \pi}(1-\rho_{\bm \pi})
\end{equation}
and write
\begin{equation} \label{eq:xi}
\xi_{\bm \pi} = 1 - f_{\bm \pi}(1-\rho_{\bm \pi}).
\end{equation}
Under Assumptions ({\bf A1}) and ({\bf A2}), the giant connected component of $\mathscr{C}( \mathbf{d}^ {(N)})$ has size $\xi_{\bm \pi} N + o(N)$ with probability tending to $1$ as $N \to \infty$.

Finally, we make two additional technical assumptions on the degree sequences $( \mathbf{d}^ {(N)})_N$:
\begin{itemize}
\item The following convergence holds:
\[   \lim\limits_{ N \rightarrow +\infty}   \frac{{d_1^{(N)}}^2 + \cdots + {d_N^{(N)}}^2}{N} = \sum\limits_{k \geq 0} k^2 \bm \pi(\{k\}).  \tag{{\bf A3}} \]
\item There exists $\gamma >2$ such that:
\[    \max \left\{ d_1^{(N)}, \ldots, d_N^{(N)}  \right\} \leq N^{1/\gamma} . \tag{{\bf A4}} \]
\end{itemize}

\bigskip

We are now ready to state our main result:

\begin{theorem} \label{theo:main}
Let $\bm \pi$ be a probability measure on $\mathbb{Z}_+$ with generating series $f_{\bm \pi}$ and $( \mathscr{C}(\mathbf{d}^{(N)}) )_{ N \geq 1}$ be a configuration model with supercritical asymptotic degree distribution $\bm \pi$ satifsying assumptions {\bf (A1)} {\bf (A2)}, {\bf (A3)} and {\bf (A4)}. Denote by $\mathcal{H}_N$ be the length of the longest induced cycle or induced path in $\mathscr{C}(\mathbf{d}^{(N)})$.

Let $\alpha_c$ be the smallest positive solution of the equation
\[
\frac{ f_{\bm \pi}^{''} \left( f_{\bm \pi}^{-1} (1 - \alpha)  \right) }{f_{\bm \pi}^{'} (1)} = 1.
\]
For every $\alpha \in [0,\alpha_c]$ and $s \in [0,1]$, we define the following functions:
\begin{align*}
g(\alpha,s) = \frac{1}{1-\alpha} f_{\bm \pi} \left( f_{\bm \pi}^{-1} (1 -\alpha) - (1-s) \frac{f_{\bm \pi}^{'} \left(  f_{\bm \pi}^{-1} (1- \alpha) \right) }{f_{ \bm \pi}'(1)}    \right)
\quad \text{and} \quad
\hat{g}(\alpha,s) = \frac{ \partial_s g(\alpha,s) }{\partial_s g(\alpha,1)}.
\end{align*}
There exists an implicit function $\alpha (\rho)$ defined on $[0,\rho_{\bm \pi}]$ such that $1-\rho = \hat{g}(\alpha(\rho),1-\rho)$.

Then,
\[ \forall \varepsilon>0, \quad \mathbb{P}\left(  \frac{\mathcal{H}_N}{N}  \geq
 \displaystyle{\int_0^{\rho_{\bm \pi}} \frac{u \, \alpha'(u)}{\partial_s \hat g (\alpha(u),1)} \mathrm{d}u}  - \varepsilon \right)  \underset{ N \rightarrow +\infty}{\longrightarrow} 1. \]
In addition, this bound still holds if we condition the graphs $\mathscr{C}(\mathbf{d}^{(N)})$ to be simple by standard arguments.
\end{theorem}

\bigskip

Although the formulation of Theorem \ref{theo:main} is implicit, explicit computations are easy in specific models. Indeed, for $d$-regular random graphs with $d \geq 3$ we find:
\[ \forall \varepsilon>0, \quad \mathbb{P}\left( \frac{\mathcal{H}_N}{N}  \geq \frac{d}{2(d-1)} \left( 1 - \int_0^1 \left(  \frac{1 - x^{\frac{1}{d-1}}}{1-x}  \right)^{\frac{2}{d-2}}  \mathrm{d}x - \varepsilon \right) \right)  \underset{ N \rightarrow +\infty}{\longrightarrow} 1 . \]
For Erd\H{o}s-Rényi random graphs with connection probability $c/N$ ($c>1)$, denoting by $\rho_c$ the smallest positive solution of $1 - \rho_c = \exp ( -c \rho_c )$, we find:
\begin{align*}
\forall \varepsilon>0, \quad  \mathbb{P} \left( \frac{\mathcal{H}_N}{N}  \geq 
\frac{\rho_c}{-\ln(1-\rho_c)} \left( \gamma +\rho_c + \ln(-\ln(1-\rho_c)) - \mathrm{Li}_2 (1-\rho_c) \right) - \varepsilon \right)  \underset{ N \rightarrow +\infty}{\longrightarrow} 1,
\end{align*}
where $\gamma \approx 0.577\ldots$ is Euler's constant and $\mathrm{Li}_2$ is the dilogarithm function.

\bigskip

To the best of our knowledge, the only other general bound on $\mathcal{H}_N$ is due to Frieze and Jackson in \cite{MR918397} in the setting where the degrees of the graph are bounded below by $3$ and uniformly bounded above but do not necessarily satisfy our assumptions. For instance their lower bound for $3$-regular graphs is approximately equal to $0.07$ while our bound is approximately equal to $0.45$. Their proof relies on a greedy algorithm which allows them to construct a macroscopic induced path in the graph. However, they do not establish the exact asymptotic length of this path, but rather a lower bound.
Building on this result they also obtained a lower bound on $\mathcal{H}_N / N$ for Erd\H{o}s-Rényi random graphs having large fixed averaged degree.

To prove Theorem \ref{theo:main}, we introduce a different algorithm, based on the depth first search algorithm, that is amenable to a detailed analysis with the framework developed in \cite{EFMN} by the authors. This allows us to exhibit an induced path with explicit macroscopic length. As we will see in Section \ref{sec:comp}, it turns out that both algorithms provide the same long induced path (and cycle) up to $o(N)$ vertices. It is worth noting that simultaneously with our work, an algorithm similar to ours and the corresponding induced paths are studied by Draganic, Glock and Krivelevich \cite{Kriv}.

\bigskip

The paper is organized as follows. In Section \ref{sec:algorithm} we present an algorithm that constructs a configuration model and spanning trees of its connected components for which ancestral lines form induced paths in the graph. We also state a result for the limiting profile of the spanning forest constructed by the algorithm in Theorem \ref{th:profile} from which Theorem \ref{theo:main} follows easily. In Section \ref{sec:analysis}, we give a detailed analysis of our algorithm using the framework of our previous works \cite{enriquez2017limiting,EFMN}. In Section \ref{sec:comp}, we show the induced path constructed by our algorithm is roughly the same as the induced path constructed by the algorithm of Frieze and Jackson \cite{MR918397}. Finally, in Section \ref{sec:mcycles}, we analyse an extension of our algorithm. This extension mixes depth-first and breadth-first explorations and constructs $m-$induced paths for any fixed $m$.

\bigskip

\noindent \textbf{Acknowledgments:}
The first author would like to thank the ANR grants MALIN  (Projet-ANR-16-CE93-0003) and PPPP (Projet-ANR-16-CE40-0016) for their financial support. The other three authors would like to thank the ANR grant ProGraM (Projet-ANR-19-CE40-0025) for its financial support. G.F. and L.M. also ackowledge the support of the Labex MME-DII (ANR11-LBX-0023-01).

\section{Constructing the graph while discovering induced paths} \label{sec:algorithm}
We now introduce an algorithm which, from the knowledge of the sequence of degree $\mathbf{d}^{(N)}$, simultaneously constructs a configuration model $\mathscr{C}(\mathbf{d}^{(N)})$ together with an exploration of it. The latter exploration, which is a modification of the depth-first exploration, is designed to discover long induced paths in the graph. At each step $n$ of the construction, we consider the following objects, defined by induction:
\begin{itemize}
\item $A_n$, the set of active vertices, is an ordered list of pairs $(v,\mathbf{m}_v)$ where $v$ is a vertex of $\mathrm{V}_N$ and $\mathbf{m}_v$ is an ordered list of elements of the form $(u,(u^1, \ldots,u^l))$ where $u$ is a vertex that will be matched to $v$ and $u^1, \ldots,u^l$ vertices that will be matched to $u$ during the exploration. 
\item $S_n$, the set of sleeping vertices, which consists of vertices that do not appear in $A_n$.
\item $R_n$, the set of retired vertices, which consists of vertices that appear neither in $A_n$ nor $S_n$.
\end{itemize}
At the initial step $n=0$ of the algorithm, we choose a vertex $v$ uniformly at random and pair each of its $d_v^{(N)}$ half-edges to uniform half-edges of the graph. Denote by $v_1, \ldots,v_l$ the corresponding vertices. For each $1 \leq i \leq l$, we successively match the half-edges of $v_i$ to uniform half-edges of the graph and denote by $v_i^1, \ldots, v_i^{k_i}$ the corresponding vertices (without repeat). Let $\mathbf{m}_v = \left((v_1,(v_1^1, \ldots, v_1^{k_1})), \ldots, (v_l,(v_l^1, \ldots, v_l^{k_l})) \right)$ and set
\[
\begin{cases}
A_0 = ((v,\mathbf{m}_v)), \\
S_0= \mathrm{V}_N \setminus \{v, v_1, \ldots, v_l \}, \\
R_0=\emptyset.
\end{cases}
\]
Suppose that $A_n$, $S_n$ and $R_n$ are constructed. Three cases are possible:
\begin{enumerate}
\item If $A_n = \emptyset$, the algorithm has just finished exploring and building a connected component of $\mathscr{C}(\mathbf{d}^{(N)})$. In this case, we select $v_{n+1}$ uniformly at random inside $S_n$ and define $\mathbf{m}_{v_{n+1}}$ exactly as before, except that the matched vertices are in $S_n$. Denoting $v_{n+1}^1, \ldots, v_{n+1}^l$ the vertices of $S_n$ matched to $v_{n+1}$, we then set:
\[
\begin{cases}
A_{n+1} = (v_{n+1},\mathbf{m}_{v_{n+1}}), \\
S_{n+1} = S_n \setminus \{v_{n+1}, v_{n+1}^1, \ldots, v_{n+1}^l \}, \\
R_{n+1} = R_n.
\end{cases}
\]

\item If $A_n \neq \emptyset$ and if its last element $(v,\mathbf{m}_{v})$ is such that $\mathbf{m}_v = \emptyset$, the exploration backtracks and we set:
\[
\begin{cases}
A_{n+1} = A_n - (v, \mathbf{m}_{v}), \\
S_{n+1} = S_n \\
R_{n+1} = R_n \cup \{v\}.
\end{cases}
\]

\item If $A_n \neq \emptyset$ and if its last element $(v,\mathbf{m}_{v})$ is such that $\mathbf{m}_v \neq \emptyset $, we denote by $(v_{n+1}, (v_{n+1}^1, \ldots,\\ v_{n+1}^l))$ the first element of $\mathbf{m}_v$. In that case, the exploration goes to $v_{n+1}$ and we construct $\mathbf{m}_{v_{n+1}}$ as before using the $v_{n+1}^i$'s and the vertices of $S_n$ that are matched to them. We finally set:
\[
\begin{cases}
A_{n+1} = A_n + (v_{n+1},\mathbf{m}_{v_{n+1}}), \\
S_{n+1} = S_n \setminus \{ v_{n+1}^1, \ldots, v_{n+1}^l \}, \\
R_{n+1} = R_n.
\end{cases}
\]
\end{enumerate}

In words, this algorithm is an interpolation between depth-first and breadth-first explorations of the graph. More precisely, our procedure first constructs the $1$-neighborhood of the current vertex $v$, which consists in some vertices $v_1,\ldots,v_l$ and it then matches sequentially the half-edges of $v_1, \ldots, v_l$ that are not yet matched. In particular, this ensures that every ancestral line of the trees constructed by the algorithm is in fact an induced path in the graph, because the discovered vertices $v_i^j$, $1 \leq i \leq j$, $1 \leq j \leq k_i$, are distinct from $v_1, \ldots, v_l$.

Since each matching of half-edges is uniform during the construction, this algorithm constructs a random graph $\mathscr{C}(\mathbf{d}^{(N)})$. Furthermore, at each step $n$, the subgraph of $\mathscr{C}(\mathbf{d}^{(N)})$ induced by the vertices of $S_n$ is a configuration model. The sequence of vertices corresponding to the first component of the last element of $A_n$ provides a spanning tree of each connected component of the graph together with a contour process of each of these trees.

We will denote by $(X_n)_{0 \leq n \leq 2N}$ the concatenated contour processes of the successive covering trees constructed by the algorithm.
Our main result, Theorem \ref{theo:main}, will be an easy consequence of the following fluid limit for the process $(X_n)_{0 \leq n \leq 2N}$:

\begin{theorem} \label{th:profile}
Recall the definition of $\rho_{\bm \pi}$ and $\xi_{\bm \pi}$ given in equations \eqref{eq:rho2} and \eqref{eq:xi}, and the definition of the functions $\alpha(\rho)$, $g (\alpha , s)$ and $\hat g (\alpha , s)$ given in Theorem \ref{theo:main}.
Under assumptions {\bf (A1)}, {\bf (A2)}, {\bf (A3)} and {\bf (A4)}, the following limit holds in probability for the topology of uniform convergence: 
$$ \forall u \in [0,2], \quad \lim_{N\to \infty} \frac{X_{\lceil u N\rceil }}{N}=h(u),$$
where the function $h$ is continuous on $[0, 2]$, null on the interval $[2 \xi_{\bm \pi}, 2]$ and defined hereafter on the interval $[0,2\xi_{\bm \pi}]$.

The graph $(u,h(u))_{u \in [0, 2 \xi_{\bm \pi}]}$ can be divided into a first increasing part and a second decreasing part.
These parts are respectively parametrized for $\rho \in [0, \rho_{\bm \pi}]$ by :
\[
\begin{cases}
x^\uparrow(\rho) & := \displaystyle{\int_\rho^{\rho_{\bm \pi}}   \frac{(2-r) \, \alpha'(r)}{\partial_s \hat g (\alpha(r),1)} \mathrm{d}r},\\
y^\uparrow(\rho) & := \displaystyle{\int_\rho^{\rho_{\bm \pi}}  \frac{r \, \alpha'(r)}{\partial_s \hat g (\alpha(r),1)} \mathrm{d}r},
\end{cases}
\]
for the increasing part and
\[
\begin{cases}
x^\downarrow(\rho) := x^\uparrow (\rho) + 2 \, \left(1- \alpha(\rho) \right) \bigg( 1 -  g \big(\alpha(\rho),1-\rho \big) \bigg),\\
y^\downarrow(\rho) := y^\uparrow(\rho),
\end{cases}
\]
for the decreasing part.
\end{theorem}

Theorem \ref{theo:main} is obtained by computing the maximal value of $h$, which is given by $y^\uparrow(0)$. The proof of Theorem \ref{th:profile} is an adaptation of the article \cite{EFMN} by the authors and is the object of the next section. It relies on Wormald's differential equations method \cite{MR1384372} via the study of ladder times of the exploration and the law of the graph induced by the sleeping vertices at these times.

\section{Analysis of the algorithm} \label{sec:analysis}

The overall strategy follows the guidelines of the previous work \cite{EFMN}. In particular, we start by identifying a good event that makes possible a decomposition of the exploration at ladder times. For every $n \in \{0 , \ldots, 2N\}$, let $D_{n}^{(N)}$ be the degree of a uniform vertex in the graph induced by $S_{n}$.
For every $\varepsilon >0$ we define
\begin{equation*}
n_\epsilon = n_{\varepsilon}^{(N)} = \sup \left\{ n \in \llbracket 0, 2N \rrbracket: \, \forall m \in \llbracket 0,n \rrbracket, \, \frac{\E[D_{n}^{(N)}( D_{n}^{(N)}-1)]}{\E[D_{n}^{(N)}]} > 1 + \varepsilon \right\}.
\end{equation*}
For $n < n_\varepsilon$, the subgraphs induced by $S_n$ are all supercritical. For $0 < \delta < 1/2$, let $\mathbf G_\varepsilon = \mathbf{G}_{\varepsilon}^{(N)}(\delta)$ be the event that, for all $n < n_\varepsilon$,
\begin{itemize}
\item there is at least one connected component with size greater than $N^{1-\delta}$ in the graph induced by $S_n$;
\item there is no connected component of size between $N^\delta$ and $N^{1 - \delta}$ in the graph induced by $S_n$.
\end{itemize}
Under our assumptions {\bf (A1)}, {\bf (A2)}, {\bf (A3)} and {\bf (A4)}, we have for every $\lambda >0$,
\begin{equation}
\label{lemma: estimate good event}
\mathbb P \big( \mathbf{G}_{\varepsilon} \big) = 1 -  \mathcal O (N^{-\lambda}).
\end{equation}

\subsection{Ladder times} \label{sec:pseudo2}
Fix $\delta \in (0,1) $. Let $T_0 = 0$ and define, for $k \in \{0 , \ldots , K\}$,
\begin{equation*}
T_{k+1} := \min \left\{ i > T_k, \, X_i = k+1 \, \, \text{and} \, \, \forall i \leq j \leq i + N^\delta, \, X_j \geq k+1 \right\},
\end{equation*}
where $K$ is the last index for which this definition makes sense (i.e. the set for which the min is taken is not empty).  Of course, this sequence of times will only be useful to analyse our algorithm when $K$ is of macroscopic order, which is indeed the case on the event $\mathbf{G}_{\varepsilon}$. Indeed, as long as $T_k < n_\varepsilon$, we can define $T_{k+1}$ on the event $\mathbf G_\varepsilon$. Therefore we set
\begin{equation}
K_\varepsilon = \sup \{ k \in \llbracket 0 , K \rrbracket : T_k < n_\varepsilon \}.
\end{equation}
Thanks to \eqref{lemma: estimate good event}, we have $K_\varepsilon < K$ with probability $1 -  \mathcal O (N^{-\lambda})$.

\begin{figure}[!h]
\begin{center}
\includegraphics[scale=0.8]{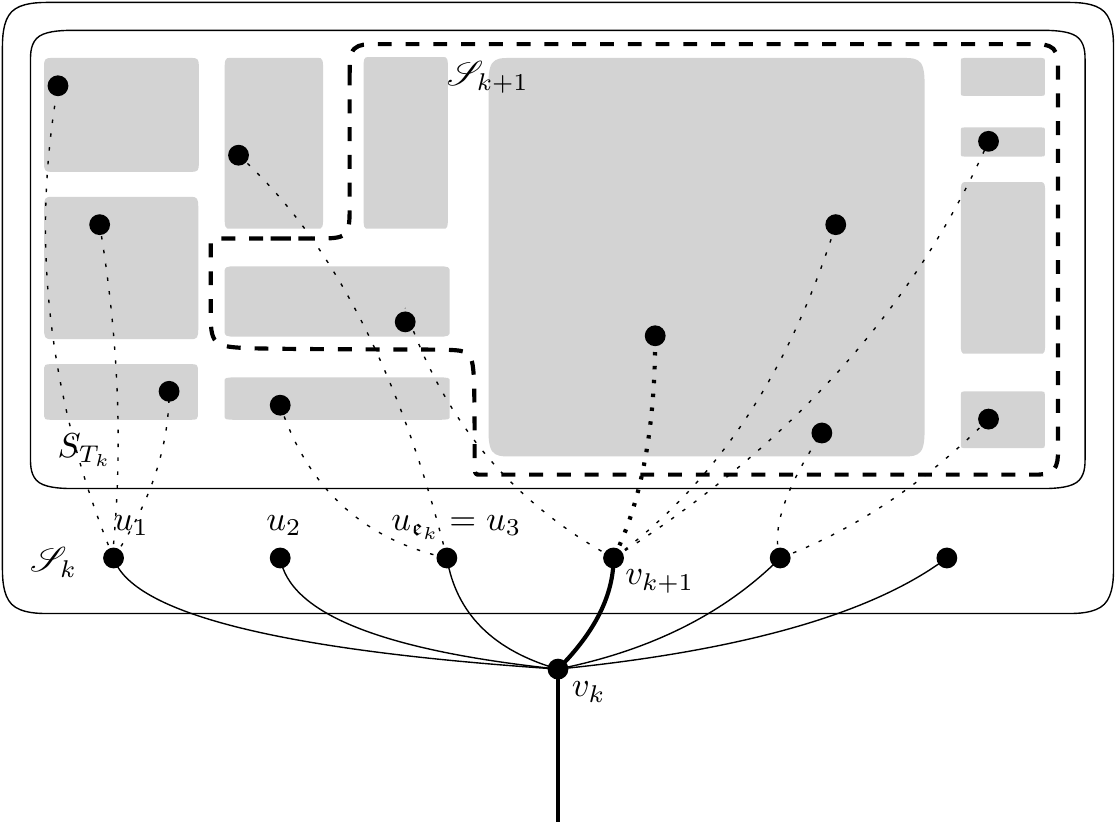}
\caption{\label{fig:ladder2} Structure of the remaining graph at a ladder time. The first half edges of $\mathbf{v}_k$ are numbered according to their matching order during the construction.}
\end{center}
\end{figure}

For all $k \in \{0 , \ldots , K\}$, let $\mathbf{v}_k$ be the vertex corresponding to the first component of the last element of $A_{T_k}$. There exists some $l\geq1$ such that the sequence $\mathbf m_{\mathbf{v}_k}$ can be written
\begin{equation} \label{eq:notationsrenewal}
\mathbf m_{\mathbf{v}_k} = \left((u_1,(u_1^1, \ldots, u_1^{k_1})), \ldots, (u_l,(u_l^1, \ldots, u_l^{k_l})) \right),   
\end{equation}
and we will denote by $\geom_k$ the index in $\{1, \ldots, l\}$ such that $u_{\geom_k +1} = \mathbf{v}_{k+1}$. We consider $\mathscr{S}_k$ the graph induced by the set of vertices $S_{T_k - 1}$. See Figure \ref{fig:ladder2} for an illustration of these definitions. As it turns out, the analysis of the structural changes of $\mathscr{S}_k$ will be crucial for our purpose. For instance, the difference between $T_{k+1} - T_k$ is equal to the time spent exploring the connected components associated to the vertices $u_1, \ldots, u_{\geom_k}$ inside the graph $\mathscr{S}_{k}$.

\subsection{Analysis of the graphs $\mathscr S_k$}

For all $k< K$, let $N_i(k)$ be the number of vertices of degree $i$ in $\mathscr{S}_k \cup \{ \mathbf{v}_k \}$ which are different from $\mathbf{v}_k$. Then, by definition of the exploration, the graph $\mathscr{S}_k$ has the law of a configuration model with vertex degrees given by the sequence $(\widetilde{N}_i(k))_{i \geq 0}$. Using the notation of \eqref{eq:notationsrenewal}, the contribution to $\widetilde{N}_i(k)$ of the edges belonging to $v_k$ is given by
\[
\sum_{p = 1}^{l} 
\left( - \mathbf{1}_{ \deg_{\mathscr{S}_k}(u_p) =i } + \mathbf{1}_{ \deg_{\mathscr{S}_k}(u_p)= i-1 } \right),
\]
and we write
\[
\widetilde{N}_i(k) = N_i(k) + \sum_{p = 1}^{l} 
\left( - \mathbf{1}_{ \deg_{\mathscr{S}_k}(u_p) =i } + \mathbf{1}_{ \deg_{\mathscr{S}_k}(u_p)= i-1 } \right).
\]
Assumption ({\bf A4}) ensures that $\widetilde{N}_i(k) - N_i(k)$ is of order $o(N)$ uniformly along the whole exploration with high probability. Henceforth, we focus our analysis on the $N_i(k)$'s whose analysis is more clear.
This evolution is ideed given by:

\begin{align}
N_i(k  +1) -  N_i (k) = 
&- V_i({\mathscr{S}_k} \setminus \mathscr{S}_{k+1}) \label{eq1} \\
& +  \sum_{p = \geom_k + 2}^{l}  \sum_{q=1}^{k_p} 
\left( - \mathbf{1}_{ \deg_{\mathscr{S}_k}(u_p^q) =i } + \mathbf{1}_{ \deg_{\mathscr{S}_k}(u_p^q)= i+1 } \right)  \label{eq4},                
\end{align}
where $V_i(S)$ stands for the number of vertices with degree $i$ in $S$ and $l$ is defined in Equation \eqref{eq:notationsrenewal}. Indeed, the first contribution corresponds to the complete removal of vertices belonging to $\mathscr {S}_k$ but not $\mathscr{S}_{k+1}$. The second contribution corresponds to the removal of edges connecting the vertices $u_{\geom_k +2}, \ldots, u_l$ to $\mathscr{S}_{k+1}$.

The crucial step of the proof is the asymptotic analysis of the variables $T_k$ and $N_i(k)$ for large $N$. This is the object of the forthcoming Theorem \ref{theo:fluidlimits}. In order to state it, we first need to introduce some notation.

Let $(z_i)_{i \geq 0} \in \mathbb{R}^{ \mathbb{Z}_+ }$ be such that $\sum_{ i \geq 0} z_i  \leq 1$ and $\sum_{k \geq 0} i z_i < \infty$. 
for any $i \geq 0 $ let $\hat{z}_i = (i+1)z_i / \sum_j j z_j$ and define:
\begin{equation}
\begin{cases} g_{(z_i)_{i \geq 0}}(s) &=  \sum\limits_{i \geq 0} \frac{z_i}{\sum_{l\geq 0} z_l} s^i \\
\hat{g}_{ (z_i)_{i \geq 0}}(s) &= \sum\limits_{i \geq 0} \hat{z}_i s^i = \frac{g'_{(z_i)_{i \geq 0} }(s)}{g'_{(z_i)_{i \geq 0} } (1)}
\end{cases}
\label{eq: DefnGenSer2}
\end{equation}
respectively the generating series associated to $(z_k)_{k \geq 0}$ and its sized biased version.
Let also $\rho_{(z_i)_{i\geq 0}}$ be the largest solution in $[0,1]$ of 
\begin{equation} \label{eq:rhog}
1-s = \hat{g}_{(z_i)_{i \geq 0} } (1-s).
\end{equation}
\begin{remark}
Since $\hat{g}$ is the generating function of a probability distribution on the integers, it is convex on $[0,1]$. Therefore, Equation \eqref{eq:rhog} has a positive solution in $(0,1]$ if and only if $\hat{g}'(1) > 1$, which is equivalent to $\frac{\sum_{l \geq 1} (l-1)l z_l}{\sum_{l \geq 1} l z_l} > 1$.
\end{remark}

We also define the following functions:

\begin{align}
f(z_0,z_1, \ldots) &=  \frac{2- \rho_{(z_i)_{i \geq 0}} }{\rho_{(z_i)_{i \geq 0}}} \label{eq:deff2} \\
f_i(z_0, z_1, \ldots) &=  - \frac{1}{\rho_{(z_j)_{j \geq 0}}} \frac{i z_i}{\sum_{j \geq 0} k z_j} +  \frac{1}{\rho_{(z_j)_{j \geq 0}}} \left( \frac{\sum_{j \geq 0} (j-1)j z_j}{ \sum_{n \geq 0} j z_j } -1 \right) \times \notag \\
& \hspace{2cm} \left( - \frac{i z_{i}}{\sum_{j \geq 0} j z_j}
 +  \frac{\sum_{j \geq 0} (j-1)j z_j}{ \sum_{n \geq 0} j z_j }  \left( \frac{(i+1) z_{i+1}}{\sum_{j \geq 0} j z_j} - \frac{i z_{i}}{\sum_{j \geq 0} j z_j} \right) \right) \label{eq:deff_i2}.
\end{align}

The asymptotic behaviour of the variables $T_k$ and $N_i(k)$ will be driven by the solution of an infinite system of differential equations whose uniqueness and existence is provided by the following lemma.

\begin{lemme}\label{lemma: EqDiff2}
Let $\bm \pi = (\bm \pi_i)_{i \geq 0} \in [0,1]^\mathbb{N}$ be a probability measure which satisfies the supercriticality assumption ({\bf A2}). Then, the following system of differential equations has a unique solution which is well defined on $[0,t_{\max})$ for some $t_{\max}>0$:
\begin{equation}\tag{S}\label{eq: EqDiff2'}
\left\{ 
\begin{array}{lcl}
\frac{\mathrm{d}z_i}{\mathrm{d}t} & = & f_i(z_0, z_1, \ldots ); \\
z_i(0) & = & \bm \pi_i.
\end{array}
\right.
\end{equation}
\end{lemme}
\begin{proof}
Let 
\begin{equation} \label{eq:F}
F(t,s):= \sum_{i \geq 0} z_i(t) s^i
\end{equation}
be the generating series associated to the $z_i$'s, which is well defined for all $s \in [0,1]$ and all $t \in [0,t_{\max}')$, where $[0,t_{\max}')$ is the maximal interval where the functions $z_i$ are well defined. Then, it is easy to check that $F$ satisfies the following partial differential equation:
\[ \frac{\partial F}{\partial t} (t,s) = \frac{1}{\rho_{ (z_j(t))_{j \geq 0} }} \frac{\frac{\partial^2 F}{\partial s^2} (t,1)}{\frac{\partial F}{\partial s}(t,1)} \frac{ \frac{\partial F}{\partial s} (t,s) }{\frac{\partial F}{\partial s} (t,1)} \left( (1-s) \frac{\frac{\partial^2 F}{\partial s^2} (t,1)}{\frac{\partial F}{\partial s} (t,1)} -1 \right). \]
Using the appropriate time change, we end up with a new generating series $f(t,s) = \sum_{j \geq 0} \zeta_j(t)s^i$ satisfying
\begin{equation}\label{eq:EqDiff}
\frac{\partial f}{\partial t} (t,s) = \frac{ \frac{\partial f}{\partial s} (t,s) }{\frac{\partial f}{\partial s} (t,1)} \left( (1-s) \frac{\frac{\partial^2 f}{\partial s^2} (t,1)}{\frac{\partial f}{\partial s} (t,1)} -1 \right).
\end{equation}
Notice that, up to a time change, this corresponds to the differential equation that already appeared in \cite{EFMN} -- see the beginning of the proof of Proposition 1. It was proved in Section 6.2 of \cite{EFMN} that it has a unique solution under our assumptions.
\end{proof}

\bigskip

We are now ready to state the main result of this section.
\begin{theorem}\label{theo:fluidlimits}
Fix $\varepsilon > 0$. With high probability, for all $k \leq K_\varepsilon$ : 
\begin{align*}
T_k &= N z\left( \frac{k}{N}  \right) + o(N) \\
N_i(k) &= N z_i \left( \frac{k}{N}   \right) + o(N),
\end{align*}
where $(z_0, z_1, \ldots)$ is the unique solution of \eqref{eq: EqDiff2'} and $z$ is the unique solution of
$\frac{\mathrm{d}z}{\mathrm{d}t} = f(z_0,z_1, \ldots)$ with initial condition given by $z(0) = 0$.

In addition, if $w(k)$ denotes the number of vertices that are not in the graph $\mathscr{S}_k$, then 
\[ w(k) = N \tilde{z}\left(\frac k N \right) + o(N), \]
where $\tilde{z}$ satisfies $\tilde{z}'(t) = \hat{g}_{(z_j(t))_{j \geq0}}'(1) / \rho_{(z_j(t))_{j \geq 0}}$ and $\tilde{z}(0)=1$.
\end{theorem}

\begin{proof}
Our main tool is an adaptation of Wormald's differential equations method, which is the content of Corollary 2 of \cite{EFMN}. To apply this result we need to check the following two points:
\begin{enumerate}
\item There exists $0 < \beta < 1/2$ such that with high probability for all $k \leq K_\varepsilon$, 
\[ |T_{k+1} - T_k | \leq N^\beta \text{   and for all   } i \geq 0,\, |N_i({k+1}) - N_i(k) | \leq N^\beta. \]
\item We denote by $(\mathcal{F}_k )_{k\geq 0}$ the canonical filtration associated to the sequence $\left( (N_i(k))_{i \geq 0} \right)_{k \geq 0} $. There exists $\lambda> 0$ such that for every $k$ and $n$, 
\begin{align} 
\E[T_{k+1} - T_k \, | \, \mathcal{F}_k] &= f \left( \frac{N_0(k)}{N} , \frac{N_1(k)}{N} , \ldots  \right) + O \left( N^{-\lambda} \right), \label{eq:Wormald1} \\
 \E[ N_i({k+1}) - N_i(k) \, | \, \mathcal{F}_k] &= f_i \left( \frac{N_0(k)}{N} , \frac{N_1(k)}{N} , \ldots  \right) + O \left( N^{-\lambda} \right) \label{eq:Wormald2}.
\end{align}
\end{enumerate}

The first point is a direct consequence of the definition of the times $T_k$'s and of Equation \eqref{lemma: estimate good event} by chosing $\delta$ small enough. 

We now turn to the second point. Since the computations are very similar to those made during the proof of Theorem 3 of \cite{EFMN}, we will only sketch them and point out the corresponding details in \cite{EFMN}. For all $k \geq 0$, let
\begin{equation} \label{eq:gk}
\begin{array}{ccccc} 
p_i = p_i(k) = \frac{N_i(k)}{\sum_{ j \geq 0} N_j(k)} & \quad & ; & \quad &  g_k = g_{(p_j)_{j\geq 0}} \\
\phantom{b} \hat{p}_i = \hat{p}_i(k) = \frac{(i+1)p_{i+1}(k)}{\sum_{j \geq 0}jp_j(k)} & \quad & ; & \quad & \phantom{l} \hat{g}_k = \hat{g}_{(p_j)_{j\geq 0}}=g_{(\hat{p}_j)_{j\geq 0}} 
\end{array},
\end{equation}
and let $\rho_k = \rho_{(p_j(k))_{j \geq 0}}$ be the largest solution in $[0,1]$ of $1-s = \hat{g}_k(1-s)$.

Recall the notation of \eqref{eq:notationsrenewal}. With high probability, the first $\geom_k$ neighbors of $\mathbf{v}_k$ belong to distinct connected components of $\mathscr{S}_k$. Denoting $W^{(1)}, \ldots, W^{(k)}$ these connected components, we deduce that
\begin{align*} 
\E[T_{k+1} - T_k \, | \, \mathcal{F}_k] 
&= 1 + 2 \mathbb{E}\left[ \sum\limits_{j=1}^k | W^{(j)} | \, \big| \, \mathcal{F}_k  \right] \\
&= \frac{2 - \rho_k}{\rho_k} + \mathcal{O}(N^{-\lambda}),
\end{align*}
where the last equality is the content of Equation (12) in \cite{EFMN}. This proves \eqref{eq:Wormald1}.

We now fix $i \geq 0$, $k \geq 0$, and prove \eqref{eq:Wormald2} by examining separately the contributions \eqref{eq1} and \eqref{eq4}. The contribution \eqref{eq1} corresponds to the removal of vertices of degree $i$ inside the small components $W^{(j)}$'s, to the removal of $\mathbf{v}_{k+1}$, and to the removal of $u_{\geom_k +2}, \ldots, u_l$. It is given by
\begin{align} 
\mathbb{E}&\left[  V_i \left(  \mathscr S_{k} \setminus \mathscr{S}_{k+1} \right) \, | \, \mathcal{F}_k \right] \notag \\
&=   \mathbb{E}\left[ V_i \left(  \cup_{j=1}^{\geom_k} W^{(j)}  \right) \, | \, \mathcal{F}_k \right] + \mathbb{P}( \deg_{\mathscr S_k}(\mathbf{v}_{k+1}) = i \, | \, \mathcal{F}_k) + \mathbb{E}\Bigg[  \sum_{j= \geom_k + 2}^l \mathbf{1}_{ \{ \deg_{\mathscr{S}_k}(u_j) = i \}} \, | \, \mathcal{F}_k \Bigg]  \notag \\
\end{align}
The terms $\mathbb{E}\left[ V_i \left(  \cup_{j=1}^{\geom_k} W^{(j)}  \right) \, | \, \mathcal{F}_k \right]$ and $\mathbb{P}( \deg_{\mathscr S_k}(\mathbf{v}_{k+1}) = i \, | \, \mathcal{F}_k)$ were respectively computed in Equations (14) and (10) of \cite{EFMN}. They are given by
\begin{align} 
\mathbb{E}\left[ V_i \left(  \cup_{j=1}^{\geom_k} W^{(j)}  \right) \, | \, \mathcal{F}_k \right] &=  \frac{\widehat{p}_{i-1}}{\rho_k} (1 - \rho_k)^{i-1} + \mathcal{O}(N^{-\lambda}) \label{eq:Contrib0} \\ 
\mathbb{P}( \deg_{\mathscr S_k}(\mathbf{v}_{k+1}) = i \, | \, \mathcal{F}_k) &= \frac{\hat{p}_{i-1}}{\rho_k} \left( 1 - (1-\rho_k)^{i-1} \right)  + \mathcal{O}(N^{-\lambda}).   \label{eq:Contrib1}
\end{align}

For the last term, we use that with high probability, $u_{\geom_k+2}, \ldots, u_l$ are distinct vertices. All of them are connected to $\mathbf{v}_k$ through a uniform matching of half-edges. Therefore:
\begin{align}
\mathbb{E}\Bigg[  \sum_{j= \geom_k + 2}^l \mathbf{1}_{ \{ \deg_{\mathscr{S}_k}(u_j) = i \}} \, \Bigg| \, \mathcal{F}_k \Bigg] 
&= \mathbb{E}\left[ \deg_{\mathscr{S}_k}(\mathbf{v}_k) - \geom_k - 1 \, | \, \mathcal{F}_k \right] \hat{p}_{i-1}  + \mathcal{O}( N^{-\lambda}) \notag \\
&=  \frac{1}{\rho_k} \left( \hat{g}_k'(1) - 1 \right) \hat{p}_{i-1} + \mathcal{O}( N^{-\lambda}). \label{eq:Contrib2}
\end{align}
where the computation of $\mathbb{E}\left[ \deg_{\mathscr{S}_k}(\mathbf{v}_k) - \geom_k -1 \, | \, \mathcal{F}_k   \right]$ can be found in Equation (15) of \cite{EFMN}.

Finally, the contribution \eqref{eq4} is given by
\begin{align}
\mathbb{E} \Bigg[ \sum_{p = \geom_k + 2}^{l}  \sum_{q=1}^{k_p}
& \left( - \mathbf{1}_{ \deg_{\mathscr{S}_k}(u_p^q) =i } +  \mathbf{1}_{ \deg_{\mathscr{S}_k}(u_p^q)= i+1 } \right)  \Bigg| \mathcal{F}_k \Bigg] \notag \\ 
&= \mathbb{E}\left[ \deg_{\mathscr{S}_k}(\mathbf{v}_k) - \geom_k - 1 \, | \, \mathcal{F}_k \right] \hat{g}_k'(1) \left( - \hat{p}_{i-1} + \hat{p}_i \right) + \mathcal{O}(N^{-\lambda}) \notag \\
&=  \frac{1}{\rho_k}\left( \hat{g}_k'(1) -1  \right)  \hat{g}_k'(1) \left( - \hat{p}_{i-1} + \hat{p}_i \right) + \mathcal{O}(N^{-\lambda})           \label{eq:Contrib4}
\end{align}
Combining \eqref{eq:Contrib0}, \eqref{eq:Contrib1}, \eqref{eq:Contrib2} and \eqref{eq:Contrib4} we obtain
\begin{align} \label{eq:m=1}
\mathbb{E}\left[N_i(k  +1) -  N_i (k) \, | \, \mathcal{F}_k  \right] 
&= \frac{\hat{g}_k'(1)}{\rho_k} \left( - \hat{p}_{i-1} \hat{g}_k'(1) + \hat{p}_i (-1 + \hat{g}_k'(1))   \right) +  \mathcal{O}(N^{-\lambda})
\end{align}
giving Equation \eqref{eq:Wormald2}. 

\bigskip

We finally turn to the last claim of Theorem \ref{theo:fluidlimits}. For all $k < K_\varepsilon$, let $w(k)$ be the number of vertices that are not in the graph $\mathscr{S}_k$. Using the notation of \eqref{eq:notationsrenewal}, the evolution of $w(k)$ is given by
\begin{align*}
w(k+1) - w(k) 
&= V( \mathscr{S}_{k} \setminus \mathscr{S}_{k+1} ) \\
&= V\left( \cup_{j=1}^{\geom_k} W^{(j)} \right) + 1 + (l - \geom_k -1).
\end{align*}
Therefore, 
\begin{equation} \label{eq:wk}
\mathbb{E}\left[ w(k+1) - w(k) \, | \, \mathcal{F}_k \right] 
= \frac{1-\rho_k}{\rho_k} + 1 + \frac{1}{\rho_k}\left(  \hat{g}_k'(1) - 1 \right) 
= \frac{1}{\rho_k} \hat{g}_k'(1). 
\end{equation}
Using Wormald's differential equations method as before we deduce that with high probability
\[ w(k) = N \tilde{z}(k/N) +o(N), \]
where $\tilde{z}$ is the only solution of $y'(t) = \hat{g}_{(z_j(t))_{j \geq0}}'(1) / \rho_{(z_j(t))_{j \geq 0}}$ with initial condition $y(0)=1$.
\end{proof}

\subsection{The asymptotic degree distribution inside $\mathscr{S}_k$}\label{sec:proportion}

As in \cite{EFMN}, Theorem \ref{theo:fluidlimits} allows us to identify the laws of the graphs induced by sleeping vertices remaining after having explored a given proportion of vertices. Remarkably, these laws are the same as the laws appearing in the Depth First Search algorithm in Theorem 1 of \cite{EFMN}. We will see in the next section that the speed at which the graph is explored is however different.

\begin{theorem}\label{theo:law2}
Recall the definition of $\alpha_c$ and of $g(\alpha,s)$ given in Theorem \ref{theo:main}.
For every $\alpha \in [0,\alpha_c]$, let $\bm \pi_\alpha$ be the probability distribution on $\mathbb{Z}_+$ with generating series $g(\alpha,s)$.
Then, for every $\alpha \in [0,\alpha_c]$, denoting $\tau^{(N)}(\alpha) = \inf \{ n \geq 1, \, |S_n | \leq (1-\alpha) N   \}$, the empirical degree distribution of the graph induced by the vertices of $S_{ \tau^{(N)}(\alpha) }$ converges to $\bm \pi_\alpha$.
\end{theorem}

\begin{proof}
Fix $\alpha \in [0, \alpha_c)$, by definition, a configuration model with asymptotic degree distribution $\bm \pi_\alpha$ is supercritical. Therefore,
\[
\frac{
\mathbb E \left[ D_{\tau^{(N)}(\alpha)} (D_{\tau^{(N)}(\alpha)}-1) \right]
}
{
\mathbb E \left[ D_{\tau^{(N)}(\alpha)} \right]
}
> 1+ \varepsilon
\]
for some $\varepsilon > 0$ and $N$ large enough, which ensures that $\tau^{(N)}(\alpha) < n_\varepsilon$.

By definition of $\tilde{z}$, this also ensures that $N \tilde{z}^{-1}(\alpha) < K_\varepsilon$ and that the proportion of explored vertices at time $T_{N \tilde{z}^{-1}(\alpha)}$ is $\alpha N + o(N)$. Therefore, $\tau^{(N)}(\alpha) = \inf\{ n \geq 0, \, |S_n| \leq (1-\alpha)N \} = T_{N \tilde{z}^{-1}(\alpha)} + o(N)$ and for all $i \geq 0$,
\begin{align*}
\left| V_i\left( S_{\tau^{(N)}(\alpha)} \right) \right|
&= N_i\left(  T_{N \tilde{z}^{-1}(\alpha)} + o(N)  \right) \\
&= N z_i \left( \tilde{z}^{-1}(\alpha) \right) + o(N).
\end{align*}
It is easy to check that the generating series associated to the sequence of functions $(z_i \circ \tilde{z}^{-1})_{ i \geq 0}$ is solution to the partial differential equation \eqref{eq:EqDiff}.
\end{proof}

\subsection{Proof of Theorem \ref{th:profile}}

Let $N\geq 1$. By definition, for all $1 \leq k \leq K_\varepsilon$, the contour process of the tree constructed by our algorithm at time $T_k$ is located at point $(T_k,k)$. Furthermore, by Theorem \ref{theo:fluidlimits}, 
\[  (T_k,k) = N \left( z\left( \frac{k}{N}\right) + o(1), \frac{k}{N} \right) .   \]
Note that $|T_{k+1} - T_k| = o(N)$ and that, between two consecutive $T_k$'s, the contour process cannot fluctuate by more than $o(N)$. Hence, after normalization by $N$, in its increasing part the limiting contour process converges to the curve $(z(t),t)$ where $t$ ranges from $0$ to $t_{\max} = \sup \{ t > 0, \, z'(t) < +\infty   \}$. In the remaining, we provide a parametrization $(x^\uparrow(\rho),y^\uparrow(\rho))_{ \rho \in (0, \rho_{ \bm \pi}]}$ of this increasing curve. Recalling the definition of $f$ given in Equation \eqref{eq:deff2}, our Theorem \ref{theo:fluidlimits} gives us
\[
\frac{(x^\uparrow)'(\rho)}{(y^\uparrow)'(\rho)} = \frac{2}{\rho}-1.
\]
exactly as in \cite{EFMN}.

To obtain a second equation involving $(x^\uparrow)'(\rho)$ or $(y^\uparrow)'(\rho)$, the paper \cite{EFMN} uses the implicit function $\alpha(\rho)$ given by the only solution of $1-\rho = \hat g(\alpha,1-\rho)$.
The link between $\alpha,x^\uparrow,y^\uparrow$ is however different in our setting. In order to establish this link, let us first notice that $\left(\rho(t) = \rho_{(z_i(t))_{i \geq 0}}\right)_t$ is the fluid limit of the survival probability $(\rho_k)_k$ as $N$ tends to infinity, where we recall that the $\rho_k$'s are defined in terms of the functions $g_k$ (see Equation \eqref{eq:gk}). Indeed, since for all $t \geq 0$, we have
\[ 1 - \rho_{\lfloor tN \rfloor} = \hat{g}_{\lfloor tN \rfloor}(1 - \rho_{\lfloor tN \rfloor}),  \]
the fact that the sequence of generating series $(\hat{g}_{\lfloor tN \rfloor} (s))_{N \geq 0}$ converges to $\partial_sF(t,s)/\partial_sF(t,1)$ (recall that $F$ is defined in Equation \eqref{eq:F}) as $N$ tends to infinity and an application of Dini's Theorem ensures that $\rho_{\lfloor tN \rfloor}$ converges to $\rho(t)$. For all $N \geq 0$, we define the function $\alpha^{(N)}$ by
\[ \forall \rho \in (0, \rho_{\bm \pi}], \quad \alpha^{(N)}(\rho) = 1 - \frac{|\mathscr{S}_{k(\rho)}|}{N} \quad \text{where} \quad
k(\rho) = 
\underset{0 \leq k \leq N}{\mathrm{argmin}} \{ | 1 - \rho - \hat{g}_{k}(1-\rho) |\}. \]
From this definition, it is clear that $\alpha^{(N)}$ converges to the implicit function $\alpha$.
Moreover, by definition of $w$ in Section \ref{sec:proportion}, we have that
\[  \forall k \in \llbracket 0 , K_\varepsilon \rrbracket, \quad w(k) = N \alpha^{(N)}(\rho_k), \]
giving that $\tilde{z}(t) = \alpha( \rho(t))$. On the other hand, using Theorem \ref{theo:fluidlimits}, we deduce that
\[  \frac{\mathrm{d}}{\mathrm{d}t}\tilde{z}(t) = \frac{1}{\rho(t)} \partial_s \hat{g}( \alpha(\rho(t)) ,1), \]
which finally yields
\begin{equation}
\label{eq:y'}
y'(\rho) =  \frac{\rho \alpha'(\rho)}{\partial_s \hat g (\alpha(\rho),1)}.
\end{equation}

The decreasing part, is obtained by translating horizontally each point $(x^\uparrow(\rho),y^\uparrow(\rho))$ of the ascending phase  to the right by  twice the asymptotic proportion of the giant component of the remaining graph of parameter $\rho$, which is $2(1-g(\alpha(\rho),1-\rho))$. Indeed, the time it takes to the DFS to return at a given height $k$ attained during the ascending phase corresponds to the time of exploration of the giant component of the unexplored graph at time $T_k$. The latter is given by twice the number of vertices of the giant component which is equal to $2 ( 1 - g_k(1 - \rho_k) )$.

\subsection{From induced paths to induced cycles}\label{sec:inducedpathcycle}

We are going to show that, with high probability, one of the first vertices of the spine of the tree constructed by our algorithm shares a common neighbor with one of the last vertices of the spine and that this common neighbor is not connected to another vertex of the spine. See Figure \ref{fig:spine} left for an illustration. This ensures that our bound on the length of the longest induced path is also valid for the longest induced cycle.

\begin{figure}[!h]
\begin{center}
\includegraphics[scale=0.8]{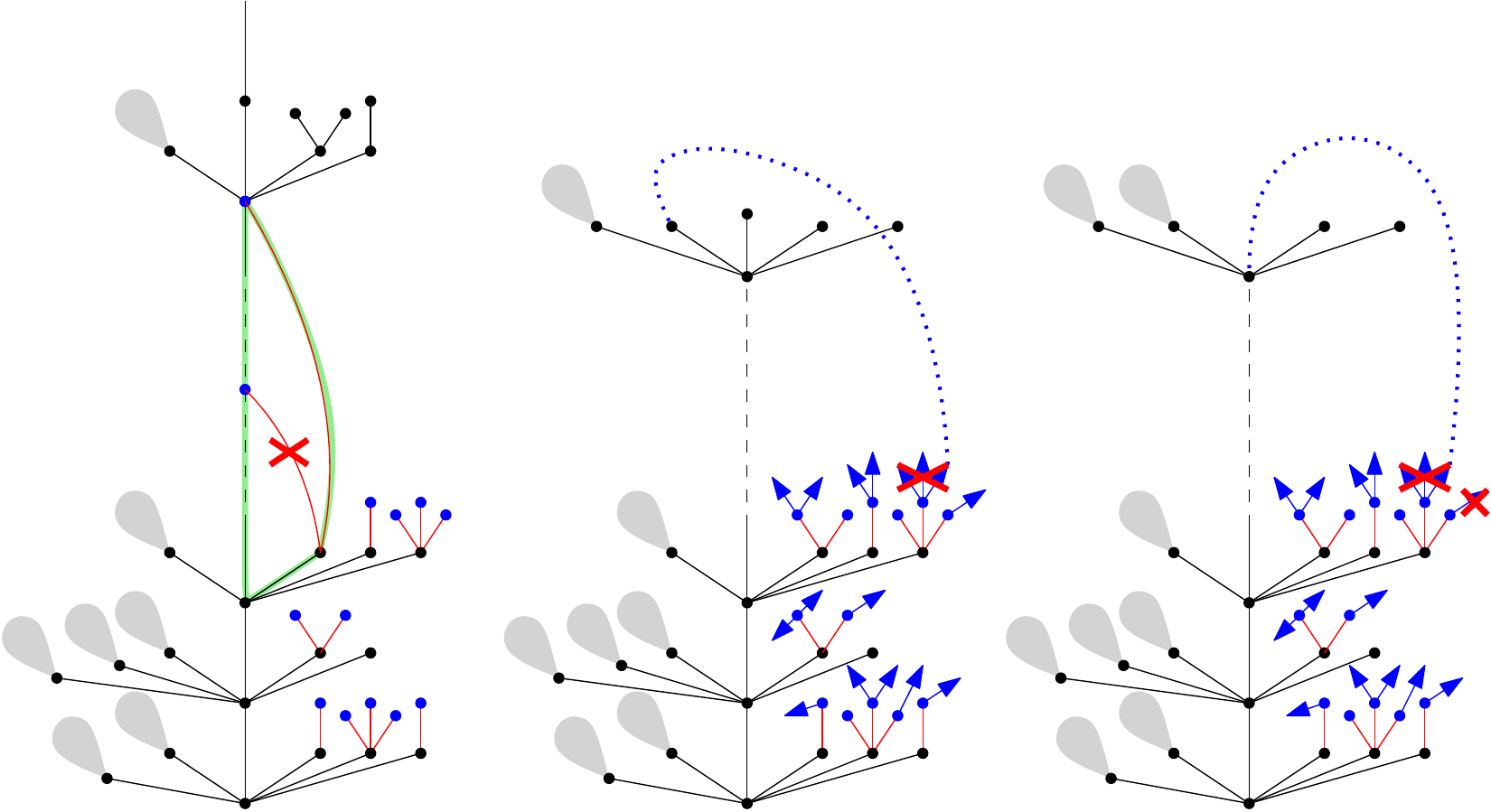}
\caption{\label{fig:spine} Left: an induced cycle constructed by the algorithm in green. Center: a matching to a half edge of $E^\varepsilon$ for a vertex that is not on the spine. Right: a matching to a half edge of $E^\varepsilon$ for a vertex of the spine.}
\end{center}
\end{figure}

Recall the definition of the ladder times $T_k$, of the associated vertices $\mathbf v_k$ on the spine of the tree and of $\mathbf m_{\mathbf v_k} = \left((u_1,(u_1^1, \ldots, u_1^{k_1})), \ldots, (u_l,(u_l^1, \ldots, u_l^{k_l})) \right)$ given in Equation \eqref{eq:notationsrenewal}. The candidates for a common neighbor between vertices of the spine are $u_{\mathfrak e_k + 2} , \ldots , u_l$ (recall that $u_{\mathfrak e_k + 1} = v_{k+1}$). These vertices are connected to $\mathbf v_k$ and to a half edge of each of the vertices $u_i^j$ for $i \in \llbracket \mathfrak e_k + 2 , l \rrbracket$. The induced cycles we are interested in are when one of these vertices $u_i^j$ is on the spine of the tree. We denote $E_k$ the set of half edges of the vertices $u_i^j$ for $i \in \llbracket \mathfrak e_k + 2 , l \rrbracket$ not yet matched. For $\varepsilon > 0$ small enough, the number of half edges of $E^\varepsilon = \bigcup_{k=1}^{\varepsilon N} E_k$ that are still unmatched at time $T_{\varepsilon N}$ is larger than $\eta N$ for some $\eta >0$. For $\varepsilon N \leq k \leq K_\varepsilon$, we denote by $\eta_k N$ the number of half edges of $E^\varepsilon$ that are still unmatched at time $T_k$ and such that the corresponding vertex has not been connected to a vertex of the spine between heights $\varepsilon N$ and $k$. The evolution of $\eta_k$ can be studied with Wormald's differential equation method in a similar fashion as in the proof of Theorem \ref{theo:fluidlimits}. We have:
\begin{align*}
\mathbb E & \left[ \eta_{k+1} N - \eta_k N \middle| \mathcal F_k \right] \\
& \geq
- \left[
\left(
\frac{1-\rho_k}{\rho_k} + \frac{1}{\rho_k} (\hat g'_k(1) - 1)
\right) \frac{\hat g_0 '(1)}{ 1 - \hat g _0(0)}
+ \left(
\frac{\hat g_0 '(1)}{ 1 - \hat g _0(0)}
\right)^2
\right]
\frac{\eta_k \, N}{|\mathscr S _k| \, g'_k(1)}
+ O \left( N^{-\lambda} \right).
\end{align*}
Indeed, each new matching of the algorithm has a probability $\frac{\eta_k \, N}{|\mathscr S _k| \, g'_k(1)}$ to be with an available half edge of $E^\varepsilon$. The first term $\left(
\frac{1-\rho_k}{\rho_k} + \frac{1}{\rho_k} (\hat g'_k(1) - 1)
\right) \frac{\hat g_0 '(1)}{ 1 - \hat g _0(0)}$ corresponds to the number of matchings along the exploration between $T_k$ and $T_{k+1}$, except for the matching of $\mathbf v_k$ and $\mathbf v_{k+1}$ (see Figure \ref{fig:spine}, center). For such matchings, we have to remove every other half edge of the discovered vertex from $E^\varepsilon$, whose expected number is roughly the expectation of $\hat \pi$ conditioned to be nonzero since it was fixed at the very begining of the construction (note that these half edges are still available for the construction).
If $\mathbf v_k$ and $\mathbf v_{k+1}$ are matched by a half edge of $E^\varepsilon$, this means that $v_{k+1}$ is a neighbor of a vertex $u'$ at distance $1$ from the first $\varepsilon N$ first vertices of the spine. In this case, we have to remove from $E^\varepsilon$ every half edge of $\mathbf v_{k+1}$ and every half edge of $E^\varepsilon$ belonging to a vertex connected to $u'$ (see Figure \ref{fig:spine}, right). The expectation of the number of these removed edges is bounded from above by the square expectation of $\hat \pi$ conditioned to be non $0$ (since some some half edges may have been matched previously).

As long as $k \leq K_\varepsilon$, the factor in front of $\eta_k$ in the above equation is uniformly bounded in absolute value by a finite constant $c_\varepsilon$ depending of $\varepsilon$. This ensures that $(\eta_{\lfloor t N \rfloor})_t$ is stochastically dominated by a process that has a fluid limit that is smaller than the solution of the differential equation $y' = - c_\varepsilon y$ as long as $t < K_\varepsilon /N$, thus $\eta_{K_\varepsilon} > \eta \exp(-c_\varepsilon t_{max})$ with high probability.

We established that at the ladder time $T_{K_\varepsilon}$, with high probability, at least $\eta \exp(-c_\varepsilon t_{max}) N$ available half edges that belong to some vertex $u$ are such that:
\begin{itemize}
\item The vertex $u$ is connected to a vertex $u'$ which is itself connected to the spine at height lower than $\varepsilon N$.
\item The vertex $u'$ is not connected to any vertex of the spine between heights $\varepsilon N + 1$ and $K_\varepsilon$.
\end{itemize}
For $k > K_\varepsilon$, each vertex of the spine $\mathbf v_k$ has a probability roughly $\frac{\eta \exp(-c_\varepsilon t_{max})}{|\mathscr S _k| / N} = \mathcal O (1)$ to form with $E^\varepsilon$ an induced cycle of length at least $K_\varepsilon - \varepsilon N$. Taking first the limit $N\to \infty$ and then $\varepsilon \to 0$ proves that our algorithm constructs an induced cycle of the same macroscopic length as the spine.

\section{Comparison with Frieze and Jackson's algorithm} \label{sec:comp}

In \cite{MR918397}, Frieze and Jackson study the following variant of the Depth-First algorithm that constructs a subtree of each connected component of the graph: Perform a depth-first exploration of the graph, and each time a new vertex is discovered, ask if it is connected to a vertex belonging to its ancestral line in the exploration tree. If it is the case, delete this vertex from the exploration tree and backtrack. The ancestral lines of trees constructed by this algorithm are induced paths by construction, however they are not necessarily spanning trees of the corresponding connected components.

Our algorithm and Frieze and Jackson's algorithm do not, in general, provide the same induced paths on some deterministic examples of graphs. However, for configuration graphs satifying assumptions {\bf (A1), (A2), (A3), (A4)}, with high probability, they will construct two trees with identical spines up to a microscopic number of vertices at the top. To see this, we first construct the graph and our tree with the algorithm of Section \ref{sec:algorithm}, and then perform Frieze and Jackson's algorithm on the resulting graph and starting at the same initial vertex of the giant component. The construction of the graph gives a total order on its vertices according to their first appearance in the contour of the spanning tree. We use this order to choose between neighbors in the DFS performed by Frieze and Jackson.

Fix $\varepsilon >0$. The following statement on the Frieze and Jackson's algorithm can be proved by induction. With high probability, for all $k < K_\varepsilon$,
\begin{itemize}
\item The exploration goes from $\mathbf v_k$ to $\mathbf v_{k+1}$,
\item The vertices explored between the first visit of $\mathbf v_k$ and the first visit of $\mathbf v_{k+1}$ form a subset of the vertices explored between the times $T_k$ and $T_{k+1}$ of our algorithm.
\end{itemize}
Indeed, if Frieze and Jackson's algorithm visits $\mathbf v_k$, at the first visit of this vertex, the vertices that have not been explored by this algorithm are the vertices of $\mathscr S_k$, a subset of the vertices that our algorithm has explored before time $T_k$, and the vertices attached to the spine before $\mathbf v_k$. The vertices already explored by our algorithm that are still available for Frieze and Jackson's algorithm belong to small connected components inside past graphs $\mathscr S_i$ for some $i<k$ and will never be explored by Frieze and Jackson's algorithm. The vertices attached to the spine form cycles in the graph. They are therefore deleted if explored by Frieze and Jackson's algorithm. Thus, after the first visit of $\mathbf v_k$, the vertices that are truly available for Frieze and Jackson's algorithm are the vertices of $\mathscr S_k$ and the induction follows.

Finally, from $\mathbf v_{K_\varepsilon}$, Frieze and Jackson's algorithm explores a subgraph of $\mathscr S_{K_\varepsilon}$ and the length of the induced path cannot be increased by more than the size of its giant component. This giant component has a size $C_\varepsilon N$ with $C_\varepsilon \to 0$ as $\varepsilon \to 0$.

\section{Extension to $m$-induced cycles} \label{sec:mcycles}
Let $m \geq 1$. An $m$-induced cycle inside a given graph $\mathrm{G}$ is a cycle such that two vertices separated by $k$ edges of the cycle have a distance at least $\min \{m,k\}$ in the graph $\mathrm{G}$. When $m=1$, we retrieve the definition of induced cycles. In this last section, we briefly explain how to adapt our arguments in order to find the following lower bound on the length of the longest $m$-induced cycle in a configuration model:

\begin{prop} \label{prop:mspine}
Let $\mathcal H_N^{m}$ be the length of the longest $m$-induced cycle in in $\mathscr{C}(\mathbf{d}^{(N)})$ satisfying assumptions {\bf (A1), (A2), (A3), (A4)}. Then, 
\[ \forall \varepsilon>0, \quad \mathbb{P} \left(  \frac{\mathcal{H}_N^m}{N}  \geq m 
 \displaystyle{\int_0^{\rho_{\bm \pi}} \frac{u \, \alpha'(u)}{\sum_{j=1}^m (\partial_s \hat g (\alpha(u),1))^j} \mathrm{d}u} -\varepsilon  \right)  \underset{N \rightarrow +\infty}{\longrightarrow} 1  . \]
\end{prop}

\begin{proof}
This bound comes from the following generalization of the algorithm defined in Section \ref{sec:algorithm}. The idea is to interpolate between a depth-first and a breadth-first exploration. Each time the exploration goes to a new vertex $v$, it discovers its $m$-neighborhood by a breadth-first exploration creating a rooted plane tree $\mathcal T_v$ with root $v$ and height at most $m$. We denote by $v_1, \ldots ,v_l$ the vertices of height $m$ in $\mathcal T_v$ listed in lexicographic order (that is in order of their discovery).
Similarly as for in Section \ref{sec:algorithm}, for each $1 \leq i \leq l$, we successively match the half-edges of $v_i$ to uniform half-edges of the graph and denote by $v_i^1, \ldots, v_i^{k_i}$ the corresponding vertices and write $\mathbf{m}_v = \left((v_1,(v_1^1, \ldots, v_1^{k_1})), \ldots, (v_l,(v_l^1, \ldots, v_l^{k_l})) \right)$. The evolution of active and sleeping vertices follows the same rules as before except that when a new vertex $v$ is explored, every vertex of the tree $\mathcal T_v$ is removed from the set of sleeping vertices.

\bigskip

\begin{figure}[!h]
\begin{center}
\includegraphics[scale=1]{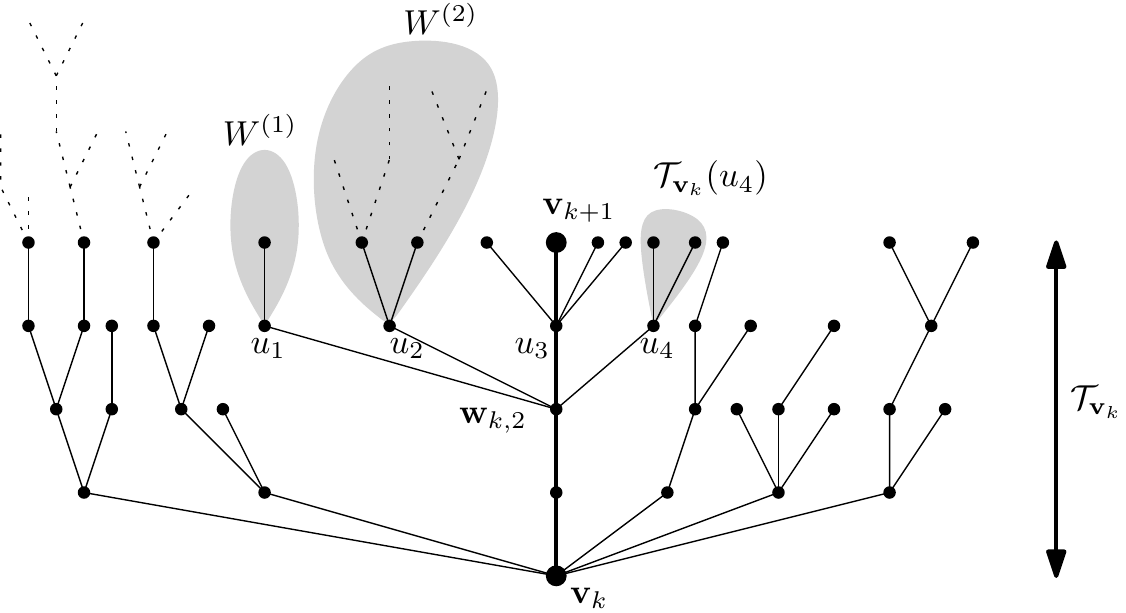}
\caption{\label{fig:m-spine} Illustration of the proof of Proposition \ref{prop:mspine} for $m=4$ and $m'=2$. Dotted lines are explored between times $T_k$ and $T_{k+1}$.}
\end{center}
\end{figure}

We can define similar ladder times $(T_k)$ for when the exploration discovers a vertex $\mathbf{v}_k$ of the giant component of the graph induced by the sleeping vertices. We denote by $\mathbf w _{k, 0} =  \mathbf{v}_k ; \mathbf w _{k, 2} ; \ldots ; \mathbf w _{k, m} =  \mathbf{v}_{k+1}$ the ancestral line between $\mathbf{v}_k$ and $\mathbf{v}_{k+1}$. Between $T_k$ and $T_{k+1}$, the algorithm explores completely the connected components of the vertices on the left hand side of this ancestral line, and up to a distance between $m$ and $1$ for the right hand side (see Figure \ref{fig:m-spine} for an illustration).
Fix $m'$ between $0$ and $m-1$ and denote $u_1, \ldots, u_l$ the children of $\mathbf w _{k, m'}$ in $\mathcal{T}_{\mathbf v_k}$. Let $\geom_{m'}$ be such that $u_{\geom_{m'} +1} = \mathbf w _{k, m'+1}$. For $j \leq \geom_{m'}$, we also denote by $W^{(j)}$ the connected components of $u_j$ explored by the algorithm. Finally, for $j > \geom_{m'} +1$, we denote by  $\mathcal{T}_{\mathbf v_k}(u_{j})$ the subtree of $\mathcal{T}_{\mathbf v_k}$ emanating from $u_j$.
The contribution of vertices attached to $\mathbf w _{k, m'}$ to the evolution of vertices of degree $i$ is given by:
\begin{align*}
\Delta_i (\mathbf w _{k, m'}) =
& - V_i\left( \cup_{j=1}^{\geom_{m'}} W^{(j)} \right) - \mathbf{1}_{\{ \deg(w _{k, m'+1})=i \}}  \\
& - \sum\limits_{ x \in \mathcal{T}_{\mathbf v_k}(u_{\geom_{m'} +2}) \cup \cdots \cup \mathcal{T}_{\mathbf v_k}(u_l)} \mathbf{1}_{ \deg(x)=i} \\
& +  \sum\limits_{ x \in \partial \mathcal{T}_{\mathbf v_k}(u_{\geom_k +2}) \cup \cdots \cup \partial \mathcal{T}_{\mathbf v_k}(u_l)}
\left( - \mathbf{1}_{ \deg(x) =i } + \mathbf{1}_{ \deg(x)= i+1 } \right) .                
\end{align*}

Computations analogous to those made during the proof of Theorem \ref{theo:fluidlimits} yield
\begin{align*}
\mathbb{E}\left[\Delta_i(\mathbf w _{k, m'}) \, | \, \mathcal{F}_k  \right] = 
& -\frac{\hat{p}_{i-1}}{\rho_k} \\
& - \frac{\hat{p}_{i-1}}{\rho_k}\left( \hat{g}_k'(1) -1 \right) \left[ \sum\limits_{j=0}^{m-m'-1} (\hat{g}_k'(1))^j  \right] \\
& - \frac{1}{\rho_k}\left( \hat{g}_k'(1) -1 \right) (\hat{g}_k'(1))^{m-m'} \left( \hat{p}_{i-1} - \hat{p}_i  \right)\\
& + \mathcal O (N^{-\lambda}).
\end{align*}
Simplifying the above expression, we obtain an equation similar to \eqref{eq:m=1} :
\begin{align} 
\mathbb{E}\left[\Delta_i(\mathbf w _{k, m'}) \, | \, \mathcal{F}_k  \right] 
&= \frac{\hat{p}_{i-1}}{\rho_k} \left(  - (\hat{g}_k'(1))^{m-m'+1}  \right) + \frac{\hat{p}_i}{\rho_k} (\hat{g}_k'(1))^{m-m'} \left( -1 + \hat{g}_k'(1) \right) + \mathcal O (N^{-\lambda})\notag \\
&= \frac{(\hat{g}_k'(1))^{m-m'}}{\rho_k} \left( - \hat{p}_{i-1} \hat{g}_k'(1) + \hat{p}_i (-1 + \hat{g}_k'(1))   \right) + \mathcal O (N^{-\lambda}). \notag
\end{align}
Summing for $m'$ between $0$ and $m-1$ finally gives:
\begin{align} 
\mathbb{E}\left[N_i(k+1) - N_i(k) \, | \, \mathcal{F}_k  \right] 
&= \frac{\sum_{j=1}^{m}(\hat{g}_k'(1))^{j}}{\rho_k} \left( - \hat{p}_{i-1} \hat{g}_k'(1) + \hat{p}_i (-1 + \hat{g}_k'(1))   \right) + \mathcal O (N^{-\lambda}). \notag
\end{align}

Using the same arguments as in the proof of Theorem \ref{theo:fluidlimits}, we deduce that for all $i \geq 0$, the function $t \rightarrow N_i(\lfloor tN \rfloor) / N$ converges pointwise towards a function $z_i$ with high probability. The corresponding sequence of functions $(z_i)_{i \geq 0}$ is the unique solution of an infinite system of differential equations, and their generating series $F(t,s)=\sum\limits_{i \geq 0} z_i(t) s^i$ satisfies the following equation:
\[ \frac{\partial F}{\partial t} (t,s) = \frac{1}{\rho_{ (z_j(t))_{j \geq 0} }} 
\left(
\sum_{m'=1}^m
\left( \frac{\frac{\partial^2 F}{\partial s^2} (t,1)}{\frac{\partial F}{\partial s}(t,1)} \right)^{m'} 
\right)
\frac{ \frac{\partial F}{\partial s} (t,s) }{\frac{\partial F}{\partial s} (t,1)} \left( (1-s) \frac{\frac{\partial^2 F}{\partial s^2} (t,1)}{\frac{\partial F}{\partial s} (t,1)} -1 \right). \]
Up to a time change, this is exactly Equation \ref{eq:F}. As it turns out, the resulting new time scale corresponds to the proportion of explored vertices during the exploration of the graph whose derivative is given by the analog of Equation \ref{eq:wk} which amounts here to the prefactor:
\[
\frac{1}{\rho_{ (z_j(t))_{j \geq 0} }} 
\left(
\sum_{m'=1}^m
\left( \frac{\frac{\partial^2 F}{\partial s^2} (t,1)}{\frac{\partial F}{\partial s}(t,1)} \right)^{m'} 
\right).
\]

\bigskip

The contour process of the spanning tree constructed by this new algorithm has therefore a fluid limit with two parametrized arcs $(x^\uparrow(\rho), y^\uparrow(\rho))$ and $(x^\downarrow(\rho), y^\downarrow(\rho))$ as in Theorem \ref{th:profile2}. The derivative of $y^\uparrow$ is given by
\[
(y^\uparrow)'(\rho) = m \cdot 
\frac{\rho \, \alpha'(\rho)}{\sum_{j=1}^m (\partial_s \hat g (\alpha(\rho),1))^j}
\]
which is the analog of Equation \eqref{eq:y'} in the current setting. Notice the factor $m$ which comes from the fact that the ancestral line between two vertices $\mathbf v_k$ and $\mathbf v_{k+1}$ has length $m$.
\end{proof}

\bibliographystyle{plain}
\bibliography{ConfigDFS.bib}

\bigskip

\noindent \textsc{Nathana\"el Enriquez:} \verb|nathanael.enriquez@math.u-psud.fr|,

\noindent \textsc{LMO, Univ. Paris Sud, 91405 Orsay France}

\medskip

\noindent \textsc{Gabriel Faraud:} \verb|gabriel.faraud@u-paris10.fr|,

\noindent \textsc{Modal'X, UPL, Univ. Paris Nanterre, F92000 Nanterre France}

\medskip

\noindent \textsc{Laurent M\'enard:} \verb|laurent.menard@normalesup.org|,

\noindent \textsc{Modal'X, UPL, Univ. Paris Nanterre, F92000 Nanterre France and New York University Shanghai, 200122 Shanghai China}

\medskip

\noindent \textsc{Nathan Noiry:} \verb|noirynathan@gmail.com|

\noindent \textsc{Telecom Paris, 91120 Palaiseau, France}

\end{document}